\documentclass[letterpaper, 10pt, conference]{ieeeconf}      

\IEEEoverridecommandlockouts                              
\usepackage{graphicx}
\usepackage{amssymb}
\usepackage{epstopdf}
\usepackage{epsfig} 
\usepackage{subfig} 
\usepackage{amsmath} 
\usepackage{color}
\usepackage{enumerate}

\title{Switching Strategies for Linear Feedback Stabilization with Sparsified State Measurements}
\author{Kang Kang, Sourabh Bhattacharya, Tamer Ba\c{s}ar}
\date{}                                           

\author{Kang Kang\quad \quad Sourabh Bhattacharya\quad \quad Tamer Ba\c{s}ar
\thanks{The authors are with the Coordinated Science Laboratory, University of Illinois at Urbana-Champaign, Urbana, IL 61801 
        {\tt\small \{kkang7,sbhattac,basar1\}@illinois.edu}}%
\thanks{This work was supported in part by a grant from AFOSR.}
}

\begin{document}
\maketitle
\begin{abstract}
In this paper, we address the problem of stabilization in continuous time linear dynamical systems using state feedback when compressive sampling techniques are used for state measurement and reconstruction. In \cite{bhatta}, we had introduced the concept of using $l_{1}$ reconstruction technique, commonly used in sparse data reconstruction, for state measurement and estimation in a discrete time linear system. In this work, we extend the previous scenario to analyse continuous time linear systems.

   We investigate the effect of switching within a set of {\it{sparsifiers}}, introduced in \cite{bhatta}, on the stability of a linear plant in continuous time settings. Initially, we analyze the problem of stabilization in low dimensional systems, following which we generalize the results to address the problem of stabilization in systems of arbitrary dimensions.
\end{abstract}
\section{Introduction}

In the last few years, compressive sensing (CS) has emerged as a topic of immense interest in the signal processing, information theory and machine learning communities \cite{candesicm},\cite{rgbcs}. The overarching principle  of CS states that by using random, non-adaptive projections, one can accurately reconstruct special signals (sparse or compressible) using far fewer measurements than what is required using the traditional techniques suggested by Nyquist and Shannon. Therefore, compressive sensing promises to provide a more efficient sensing technique than the existing ones. Reference \cite{rgbweb} provides an extensive collection of papers related to the theory behind compressive sensing, and its applications. The single-pixel camera \cite{wak1}, feature-specific imager \cite{fssimage} and CMOS separable transform image sensor \cite{gtimage} are examples of real sensors that work on principles of compressive sampling and sparse acquisition. This paper investigates the implications of using such sensors to provide feedback in dynamical systems.

Recently, there has been a growing interest in the control community regarding compressive sensing. In \cite{masaaki}, the authors investigate a feed-forward system in which a compressive sampling system is used to compress the control signals using sparse representations. In \cite{masaaki1}, a networked control architecture is improved by using compressive sensing to provide more robust packet delivery over a lossy channel. In \cite{naga1}, a compressive sensing method was used to improve bandwidth limitations on a remote-controlled system. These previous papers address utilizing compressive sensing algorithms to improve pre-existing control architecture. In \cite{daiyuk}, the authors address the problem of recovering an initial state under sparsity constraints. They provide sufficient conditions on the number of available observations in order to recover the initial state for both deterministic and stochastic systems. In a similar vein, the authors in \cite{wak2} illustrate the technique of recovering initial states from sparse observations by using a simple diffusion system.

In this work, we extend the analysis in \cite{bhatta} which investigates the use of sensors based on CS techniques for providing feedback in control systems. Earlier work \cite{bhatta} idealized the notion of sensors using CS techniques, and introduced an abstract concept of compressive sensing device (CSD). A CSD is a sensor that measures the state of a control system. The error in the measurement is dependent on the sparsity of the underlying state. Based on the error bounds provided by $l_1$ reconstruction techniques \cite{candesicm}, the sparsity of the underlying state must be within a specified level for perfect reconstruction by a CSD. In order to circumvent this limitation so that the system performance does not degrade in the face of non-sparse states, the paper has proposed the use of a {\it{sparsifier}} in order to induce sparsity in the states. Furthermore, it has provided a criteria to design a linear sparsifier to achieve stabilization in a plant. It was shown that for discrete-time linear systems stabilization can be achieved if the number of unstable poles is less than the level of sparsity that the CSD can handle. The scenario in which the condition is violated had remained unaddressed; this paper is an effort in that direction. Here, we investigate the possibility of stabilizing a plant by switching within a class of sparsifiers, each of which fails to stabilize the system individually.




Stabilization of switched systems has been a classical control problem which has been studied extensively in the past 20 years. Some of the basic problems in stability of switched systems have been listed in \cite{dlib2}. Reference \cite{dlib} provides a valuable gateway into the various available techniques available for stabilizing switched systems. For some linear systems, the approach is to find a Lyapunov function which attains a negative value for possible values of the states corresponding to each subsystem. Other methods include time-dependent switching and using multiple Lyapunov functions. Reference \cite{tbtempo} considered the case in which multiple Lyapunov functions are used, and developed state-dependent switching rules for global asymptotic stability. Reference \cite{branicky} provides techniques to stabilize both switched and hybrid systems, and \cite{dlib} contains an exhaustive review of the work that has been done regarding the stability of switched systems.


The main contribution of this paper is that we use the existing theory in hybrid switched systems to circumvent the shortcomings of the system architecture proposed in \cite{bhatta}, which is based on principles of sparse reconstruction techniques. The organization of the paper is as follows. In Section 2, we present a brief primer to compressive sensing. In Section 3, we present the problem formulation based on the architecture proposed in \cite{bhatta}. In Section 4, we provide an example of a two dimensional system to illustrate the CSD stability concept. In Section 5, we generalize our results on stability to systems of arbitrary dimensions. In Section 6, we provide simulation results. Finally, we conclude in Section 7 and provide some future research directions. 


\section{Background}
In this section, we present a brief introduction to compressive sensing techniques and their relation to sparse signal reconstruction. The content in this section is at times verbatim, but a shorter version of Section 2.3 from \cite{laska} which provides an excellent, comprehensive survey of important results in compressive sensing. It is included here as background material for the sake of completeness.

In the CS framework, we acquire a signal {\bf{x}}$\in \mathbb{R}^{N}$ via linear measurements
\begin{eqnarray}
{\bf{y}} = \Phi {\bf{x + e}}
\end{eqnarray}
where $\Phi$ is an $M\times N$ measurement matrix modeling the
sampling system, {\bf{y}}$\in \mathbb{R}^{M}$ is the vector of samples acquired,
and {\bf{e}} is an $M\times 1$ vector that represents measurement errors.
If {\bf{x}} is $K$-sparse when represented in the {\it{sparsity basis}} , i.e., ${\bf{x}} =\Psi\alpha$ with $||\alpha||_0 := |\textrm{supp}(\alpha)| \leq K$, then one can acquire
only $M = O(K \log(N/K))$ measurements and still recover
the signal {\bf{x}} \cite{donohotit}, \cite{candesicm}. A similar guarantee can be obtained for approximately sparse, or {\it{compressible}}, signals. Observe that
if $K$ is small, then the number of measurements required can
be significantly smaller than the Shannon-Nyquist rate.

In \cite{candestit}, Cand$\grave{e}$s and Tao introduced the {\it{restricted isometry property}} (RIP) of a matrix $\Phi$ and established its important role in CS. From \cite{candestit}, we have the definition:
\vspace{0.1in}
\par
\noindent
{\bf{Definition 1}}. A matrix $\Phi$ satisfies the RIP of order $K$ with constant $\delta\in (0, 1)$ if
\begin{eqnarray}
(1-\delta)||{\bf{x}}||_{2}^{2}\leq ||\Phi\bf{x}||_{2}^{2}\leq (1+\delta)||{\bf{x}}||_{2}^{2} 
\end{eqnarray}
holds for all {\bf{x}} such that $||{\bf{x}}||_{0} \leq K$.

In words, $\Phi$ acts as an approximate isometry on the set
of vectors that are $K$-sparse in the basis $\Psi$. An important
result is that for any unitary matrix $\Psi$, if we draw
a random matrix $\Phi$ whose entries $\phi_{ij}$ are independent realizations
from a sub-Gaussian distribution, then $\Phi\Psi$ will
satisfy the RIP of order $K$ with high probability provided that
$M = O(K log(N/K))$ \cite{bararip}. In this paper, without any loss of
generality, we fix $\Psi= {\bf{I}}$, the identity matrix, implying that $x = \alpha$.

The RIP is a necessary condition if we wish to be able
to recover all sparse signals {\bf{x}} from the measurements {\bf{y}}. Specifically, if $||x||_{0} = K$, then $\Phi$ must satisfy the lower bound of the RIP of order $2K$ with $ \delta< 1$ in order to ensure
that any algorithm can recover {\bf{x}} from the measurements {\bf{y}}.
Furthermore, the RIP also suffices to ensure that a variety
of practical algorithms can successfully recover any sparse or
compressible signal from noisy measurements. In particular,
for bounded errors of the form $||e||_{2}\leq \epsilon$, the convex program
\begin{eqnarray}
x = \arg\min_{\theta}\|{\bf{x}}\|_{1}\quad{\text{  s.t.  }} \|\Phi{\bf{x}}-\bf{y}\|_{2}\leq \epsilon
\end{eqnarray}
can recover a sparse or compressible signal {\bf{x}}. The following
theorem, a slight modification of Theorem 1.2 from \cite{candesrip},
makes this precise by bounding the recovery error of {\bf{x}} with
respect to the measurement noise norm, denoted by $\epsilon$, and with respect to the best approximation of {\bf{x}} by its largest $K$ terms,
denoted by ${\bf{x}}_K$.

 Suppose that $\Phi\Psi$ satisfies the RIP of order $2K$
with $\delta < \sqrt{2}-1$. Given measurements of the form $y = \Phi\Psi x+e$, where $\|{\bf{e}}\|\leq \epsilon$, the solution to (3) obeys
\begin{equation}
\|\hat{{\bf{x}}}-{\bf{x}}\|_{2}\leq C_{0}\epsilon+C_{1}\frac{\|{\bf{x}}-{\bf{x_{K}}}\|}{\sqrt{K}}
\label{eqn:recon}
\end{equation}
where
\[C_{0}=\frac{4(1+\delta)}{1-(\sqrt{2}+1)\delta},\quad C_{0}=\frac{1+(\sqrt{2}-1)\delta}{1-(\sqrt{2}+1)\delta}\]

 While convex optimization techniques like (3) constitute a powerful
method for CS signal recovery, there also exist a variety
of alternative algorithms such as CoSaMP \cite{needellCoSAMP} and iterative hard thresholding (IHT) \cite{blumeIHT} that
are known to satisfy similar guarantees under slightly stronger
assumptions on the RIP constants. In this work, we assume that measurement noise is absent, i.e., $\epsilon=0$.

\section{Problem Formulation}
In this section, initially we introduce the control architecture as proposed in \cite{bhatta}. Based on this architecture, we formulate the problem statement.

As introduced in \cite{bhatta}, the CSD is an idealized device that works on the principles of $l_1$ reconstruction algorithm \cite{candesicm}. It is assumed that the reconstruction error of the CSD obeys (\ref{eqn:recon}). Therefore, if the input state is $S$-sparse, there is no measurement error i.e, the output of the CSD perfectly replicates the input. Figure \ref{FigOne} illustrates the architecture of the control loop when a CSD is used to provide linear state feedback to stabilize a plant. In \cite{bhatta}, the idea of using a sparsifier, $G$, has been proposed, which replaces some of the state measurements with zeroes instead of their true value. This increases the sparsity of the resultant state vector to the desired value so as to ensure perfect measurement using the CSD. However, we can observe that if the number of unstable poles are greater than the maximum allowable sparsity, there are instances for which stabilization cannot be ensured by linear state feedback law. An example of such an instance is as follows. Consider the following decoupled 2-dimensional system:

 \begin{figure}[h] 
 \centering
  \caption{CSD Dynamical System Framework}
  \label{FigOne}
  \includegraphics[scale=0.7]{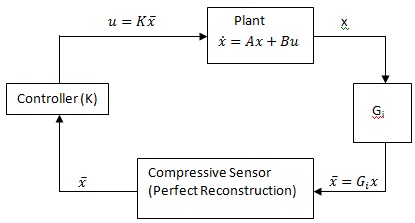}
 \end{figure}

\begin{eqnarray}
\dot{x}=
\begin{bmatrix}
1 & 0\\
0 & 2
\end{bmatrix}x+
\begin{bmatrix}
1 & 0\\
0 & 1
\end{bmatrix}u\nonumber\\
\end{eqnarray}
Let us choose the sparsifier
\(
G=
\begin{bmatrix}
1 & 0\\
0 & 0\\
\end{bmatrix}
\). Let $ K=[k_{1}\quad k_{2};k_{3}\quad k_{4}]$ be the feedback gain. The closed-loop system then has the following form:
\begin{eqnarray}
\dot{x}=
\begin{bmatrix}
1-k_1 & -k_2\\
0 & 2
\end{bmatrix}x\nonumber
\end{eqnarray}
from which it is obvious that the pole corresponding to $\lambda=2$ cannot be shifted for any value of $K$. Therefore, the system cannot be stabilized by providing linear state feedback by using the above sparsifier. This leads to the following question: Is it possible to construct a class of sparsifiers, independent of the structure of the system matrices, $A$ and $B$, for which it might be possible to stabilize the system by switching within the class? We assume that any single sparsifier in the class leads to an unstable closed-loop system. In the following, we construct a class of sparsifiers that depends on the maximum sparsity in the state that can be tolerated by the CSD for perfect reconstruction, denoted by $S$, and the dimension of the system.


 Before defining the class of sparsifiers $\mathcal{G}_S$, let us introduce a vector shift operation which we denote by $rshift(a, b)$, where $a \in \mathbb{R}^n$ and $b \in \mathbb{Z}_{+}$. $rshift(a, b)$ performs a right shift operation on each element in the vector $a$ and substitutes zeroes in the shifted slots. Let $a=[a_1, a_2, a_3, \hdots, a_n]$, then $rshift(a, b)=[0, \hdots, 0, a_{1+b}, a_{2+b}, \hdots, a_{n-b}]$ where the output of the operation is still a vector in $\mathbb{R}^n$.\\

We now define a class of sparsifiers that introduces $n-S$ zeroes in the state measurement, where $n$ is the dimension of the state space. Let the set of sparsifiers be denoted by $\mathcal{G}_S=\{G_i\}$, where $i \in \{1,2,...,\lceil {\frac{n}{S}} \rceil\}$ and $G_i \in \mathbb{R}^{n \times n}$. For each $i \in \{1,2,...,\lceil {\frac{n}{S}} \rceil\ - 1 \}$,  define each $G_i$ as a diagonal matrix with the diagonal denoted as $\mathcal{D}_i \in \mathbb{R}^n$. $\mathcal{D}_1$ is defined as follows
\[
\mathcal{D}_1 = 
\begin{bmatrix}
a_{j}
\end{bmatrix}
\]
\[
  a_j = \left\{ 
  \begin{array}{l l}
    g_i\in \mathbb{R} & \quad \textrm{if j $= \{1,2,..., \lfloor { \frac{n}{S} } \rfloor \}$}\\
    0 & \quad \textrm{else}\\
    
  \end{array} \right.
\]
For $i \in \{2,3,...,\lceil {\frac{n}{S}} \rceil\ \}$, $\mathcal{D}_i=rshift(\mathcal{D}_1, i \times \lfloor {\frac{n}{S}} \rfloor)$.\\


As an example, let us consider the case when $n=5$ and $S=2$. The set $\mathcal{G}_S$ then consists of the following set of matrices:
\begin{eqnarray}
&G_1 = 
\begin{bmatrix}
g_1 & 0 & 0 & 0 & 0\\
0 & g_1 & 0 & 0 & 0\\
0 & 0 & 0 & 0 & 0\\
0 & 0 & 0 & 0 & 0\\
0 & 0 & 0 & 0 & 0\\
\end{bmatrix}
\textrm{, }
G_2 = 
\begin{bmatrix}
0 & 0 & 0 & 0 & 0\\
0 & 0 & 0 & 0 & 0\\
0 & 0 & g_2 & 0 & 0\\
0 & 0 & 0 & g_2 & 0\\
0 & 0 & 0 & 0 & 0\\
\end{bmatrix}\nonumber\\
&\textrm{ }
G_3 = 
\begin{bmatrix}
0 & 0 & 0 & 0 & 0\\
0 & 0 & 0 & 0 & 0\\
0 & 0 & 0 & 0 & 0\\
0 & 0 & 0 & 0 & 0\\
0 & 0 & 0 & 0 & g_3\\
\end{bmatrix}\nonumber
\end{eqnarray}

 Each sparsifier in the set $\mathcal{G}_S$ introduces at least 3 ($n-S$) zeroes in the state.

Now we present the problem statement. Consider the control architecture shown in Figure \ref{FigF}. The pair $(A,B)$ is assumed to be stabilizable. Moreover, we assume that there is no $K$ such that $A_{cl}=A-BKG_{i}$ is stable for any single $G_{i}\in \mathcal{G}_{S}$. We want to find a switching scheme among the sparsifiers in the class $\mathcal{G}_{S}$ such that the system can be stabilized.  

 \begin{figure}[htb]
 \centering
 \caption{Switched System with CSD}
 \label{FigF}
 \includegraphics[scale=0.4]{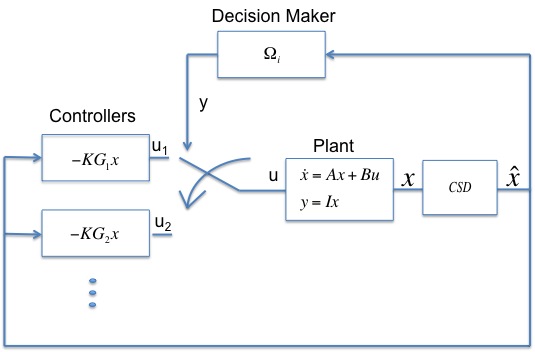}
 \end{figure}

In the next section, we investigate the problem for a two dimensional system.

\section{Two-dimensional System Example}
Let us consider a two dimensional system with sparsified feedback. In this section, we assume $B=I$ for the sake of simplicity. We define two sparsifiers in this example, and show a stability result which motivates our main theorems in the next section. 

Let $\mathcal{G}_S$ be defined as follows:
\begin{center}
$\mathcal{G}_S$:=$\{$
\(
G_1=
\begin{bmatrix}
1 & 0\\
0 & 0\\
\end{bmatrix}
\),
\(
G_2=
\begin{bmatrix}
0 & 0\\
0 & 1\\
\end{bmatrix}
\)
$\}$.
\end{center}

The closed-loop behaviour of the system is given by the following equation.
\begin{eqnarray}
\dot{x}&=&Ax+Iu=Ax-K\overline{x}\nonumber\\
&=&Ax-KG_{i}x=(A-KG_{i})x, \textrm{ } i \in (1,2) \nonumber
\end{eqnarray}

From \cite{dlib}, we know that there exists a switching scheme to guarantee stability if there exists a convex Hurwitz combination of the subsystems. We denote this convex combination by $\overline{A}$.
\begin{eqnarray}
\overline{A} = \sum_{i} \alpha_iA_i \textrm{ where } \sum_i \alpha_i = 1
\label{eqn41}
\end{eqnarray}
Next, we investigate the existence of a convex combination for a general state matrix $A$ of the following form.
\[
A = 
\begin{bmatrix}
a_{1} & a_{2} \\
a_{3} & a_{4} \\
\end{bmatrix}
\]
Let $K=kI$, where $k \in \mathbb{R}$. 
\begin{eqnarray}
A_{1} = 
\begin{bmatrix}
a_{1}-k & a_{2} \\
a_{3} & a_{4} \\
\end{bmatrix},
\quad
A_{2} = 
\begin{bmatrix}
a_{1} & a_{2} \\
a_{3} & a_{4}-k \\
\end{bmatrix}\nonumber
\end{eqnarray}

Let us choose $\alpha_i = \frac{1}{2}$ $\forall i$. Therefore, from (\ref{eqn41}) we obtain the following value of $\overline{A}$.
\[
\overline{A} =\frac{1}{2}A_1+\frac{1}{2}A_2=A-\frac{k}{2}I
\]
Let $\lambda$ be an eigenvalue of $A$, and $v$ be the corresponding eigenvector.
\begin{eqnarray}
Av=\lambda v\implies\quad k\frac{1}{2}Iv=\frac{k}{2}v\nonumber\\
\implies (A-\frac{k}{n}I)v=(\lambda-\frac{k}{2})v\nonumber\\
\overline{A}v=(\lambda-\frac{k}{2})v
\end{eqnarray}

From the above equation, we can conclude that the eigenvalues of $\overline{A}$ can be arbitrarily placed by choosing a proper $k$. For $k>2\textrm{Re}(\lambda_{max}(A))$, we obtain a Hurwitz convex combination of the two subsystems.


In the next section, we generalize the result to systems of arbitrary dimensions.

\section{Proof of Stability}

Consider the system \(\dot{x}=A_ix, i \in (1,2,...,n)\). Also let \(A_i\) be non-Hurwitz \(\forall i\).
\newtheorem{ThmOne}{Theorem}
\begin{ThmOne}
If the matrices \(A_i\) have a Hurwitz convex combination, then there exists a state-dependent switching strategy that makes the switched linear system quadratically stable \cite{dlib}.\\
\end{ThmOne}
\begin{proof}
If there \(\exists\) a convex combination such that \(\overline{A}=\sum_{i=1}^{n} \alpha_iA_i \textrm{ where } \sum_{i=1}^{n} \alpha_i = 1\) is Hurwitz, then by the definition of Lyapunov stability, for each positive definite matrix $Q$,  there exists a positive definite matrix $P$ such that
\begin{eqnarray}
\overline{A}^{T}P+P\overline{A}=-Q\nonumber
\end{eqnarray}
This implies that
\begin{eqnarray}
&&(\sum_{i=1}^{n} \alpha_iA_i)^{T}P+P(\sum_{i=1}^{n} \alpha_iA_i)=-Q\nonumber \\
\implies&& \sum_{i=1}^{n} \alpha_i(A_i^{T}P+PA_i) = -Q\nonumber \\
\implies&& x^T\sum_{i=1}^{n}\alpha_i(A_i^TP+PA_i)x=-x^TQx<0\nonumber\\
&&\forall x \in \mathbb{R}^{n}\setminus\{0\} \nonumber\\
\implies&& \exists i \in [1,n]\textrm{ s.t. }x^T\alpha_i(A_i^TP+PA_i)x<0\nonumber\\
&&\forall x \in \mathbb{R}^{n}\setminus\{0\} \nonumber
\end{eqnarray}
Let $\Omega_{i}$ be defined as follows:
\[\Omega_i:=\{x:x^T(A_i^TP+PA_i)x<0\},\quad i=[1,n]\]
Let \(x \in \mathcal{R}^n \setminus\{0\}\). Then for a given positive definite matrix $Q$, \( -x^TQx<0\).
Since there exists a convex combination of \(\overline{A}=\sum_{i=1}^{n}\alpha_iA_i\) which is Hurwitz and satisfies the following Lyapunov equation \[-Q=\overline{A}^TP+P\overline{A}\] \[\implies -Q=\overline{A}^TP+P\overline{A}=\sum_{i=1}^{n}\alpha_i(A_i^TP+PA_i)\]
 \[\implies x^T(\sum_{i=1}^{n}\alpha_i(A_i^TP+PA_i))x=-x^TQx<0\]

\[\implies\exists \textrm{ at least one } i \in \mathcal{P} \textrm{ s.t. }x^T(A_i^TP+PA_i)x<0\]
\[\implies x \in  \cup_{i \in \mathcal{P}} \Omega_i\]
On the other hand, it can trivially be shown that $x \in \cup_{i \in [1,n]} \Omega_i \implies x \in \mathcal{R}^n \setminus\{0\}$. Therefore, the following holds true:
 \[\mathbb{R}^n \setminus\{0\}\subseteq \bigcup_{i\in[1\,n]} \Omega_i\]
The switching strategy is based on keeping the system \(\dot{x}=A_ix\) active in \(\Omega_i\) since it will decrease the function \(V(x):=x^TPx\) along solutions to the differential equation.
\end{proof}

Our next goal is to show that for linear stabilizable dynamical systems, the class of sparsifiers $\mathcal{G}_S$ can always make the convex combination of subsystems in a CSD dynamical system framework Hurwitz. Before we start, let us define the following:
\begin{eqnarray}
A_i=A-\tilde{K}G_i,\quad i \in [1,.,\lceil {\frac{n}{S}} \rceil]
\end{eqnarray}
\newtheorem{ThmTwo}[ThmOne]{Theorem}
\begin{ThmTwo}
Given a stabilizable dynamical system and the class of sparsifiers $\mathcal{G}_S$, there always exists a convex combination of \(A_i\)'s, which is Hurwitz.
\end{ThmTwo}
\begin{proof}
The convex combination is defined as \(\overline{A}=\sum_{i} \alpha_iA_i \textrm{ where } \sum_{i} \alpha_i = 1\)
\begin{eqnarray}
\overline{A}&=&\sum_{i} \alpha_iA_i=\sum_{i} \alpha_i(A-B\tilde{K}G_i) \\
&=&\sum_{i}\alpha_i(A-B\tilde{K}G_i) \nonumber \\
&=&A-B\tilde{K}\Sigma \nonumber\\
\Sigma &=&\textrm{diag}(\alpha_1g_1,\hdots,\alpha_1g_1\hdots\alpha_{\lceil {\frac{n}{S}} \rceil} g_{\lceil {\frac{n}{S}} \rceil})\nonumber
\end{eqnarray}
Since the system is stabilizable, $\exists K \textrm{ such that} (A-BK)$ is stable. Moreover, if we choose $\alpha_i>0$ and $g_i>0$ $\forall i$ then $\Sigma^{-1}$ exists. Therefore, $\tilde{K}=K\Sigma^{-1}$.
\end{proof}

Based on Theorem 1, it is clear that there exists a stabilizing switching control scheme because there exists a Hurwitz convex combination of the individual subsystems.

\section{Simulation Results}
We consider the following system matrices, and the sparsifier class.
\[
A=
\left[\begin{array}{cc}
1 & 0\\
0 & 2\\
\end{array}\right]
,
B=
\left[\begin{array}{cc}
1 & 0\\
0 & 1\\
\end{array}\right]
\]
\begin{center}
$\mathcal{G}_S$:=$\{$
\(
G_1=
\begin{bmatrix}
g_1 & 0\\
0 & 0\\
\end{bmatrix}
\),
\(
G_2=
\begin{bmatrix}
0 & 0\\
0 & g_2\\
\end{bmatrix}
\)
$\}$.
\end{center}
Clearly the system (A, B) is stabilizable (in fact, controllable). A linear controller based on the gain 
\[
K=
\begin{bmatrix}
4 & 0\\
0 & 4\\
\end{bmatrix}
\]
will make $A-BK$ Hurwitz. Without loss of generality, we let $\tilde{K}=I$. Letting $\alpha_i=\frac{1}{2}$, $\forall i$, and $g_1=g_2=8$ leads us to the following:
\[
\tilde{K}\Sigma=
\begin{bmatrix}
4 & 0\\
0 & 4\\
\end{bmatrix}
\]
By Theorem 2, we can find a class of sparsifiers $\mathcal{G}_S$ that yields a Hurwitz convex combination of the subsystems. Using Theorem 1, we can find a common Lyapunov function, and a state-dependent switching strategy.

Let 
\(
Q =I 
\). 
Substituting 
\(
\overline{A} = 
\begin{bmatrix}
-3 & 0 \\
0 & -2 \\
\end{bmatrix} 
\)
in the Lyapunov equation $\overline{A}^TP+P\overline{A}=-Q\nonumber$ leads to the following value of $P$:
\[
P=
\begin{bmatrix}
1/6 & 0 \\
0 & 1/4 \\
\end{bmatrix}
\]
This creates two conic regions which determine the switching strategy between the sparsifying transforms.
\begin{eqnarray}
\Omega_{1} :&=& (x:x^{T} (A_{1}^{T} P + PA_{1})x<0)\nonumber\\
 &=& ({x:-2.333x_{1}^2+x_{2}^2<0})\nonumber\\
\Omega_{2} :&=& ({x:x^{T} (A_{2}^{T} P + PA_{2})x<0})\nonumber\\
 &=& ({x:\frac{1}{3}x_{1}^2-3x_{2}^2<0})\nonumber
\end{eqnarray}

\begin{figure}[tb]
\caption{State Trajectory}
\label{FigTwo}
\includegraphics[scale=0.6]{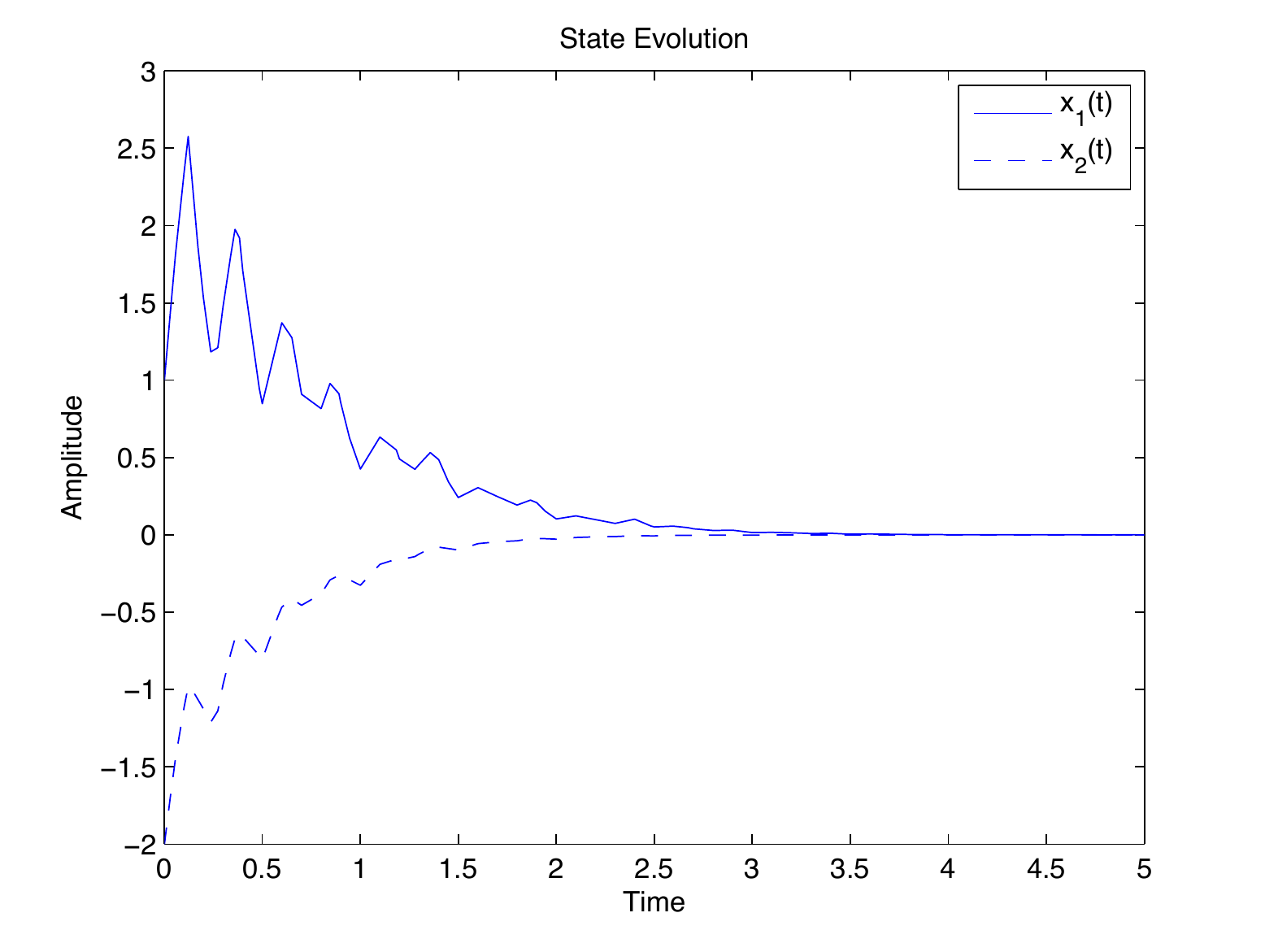}
\end{figure}

\begin{figure}[h]
\caption{Phase Portrait}
\label{FigThree}
\includegraphics[scale=0.6]{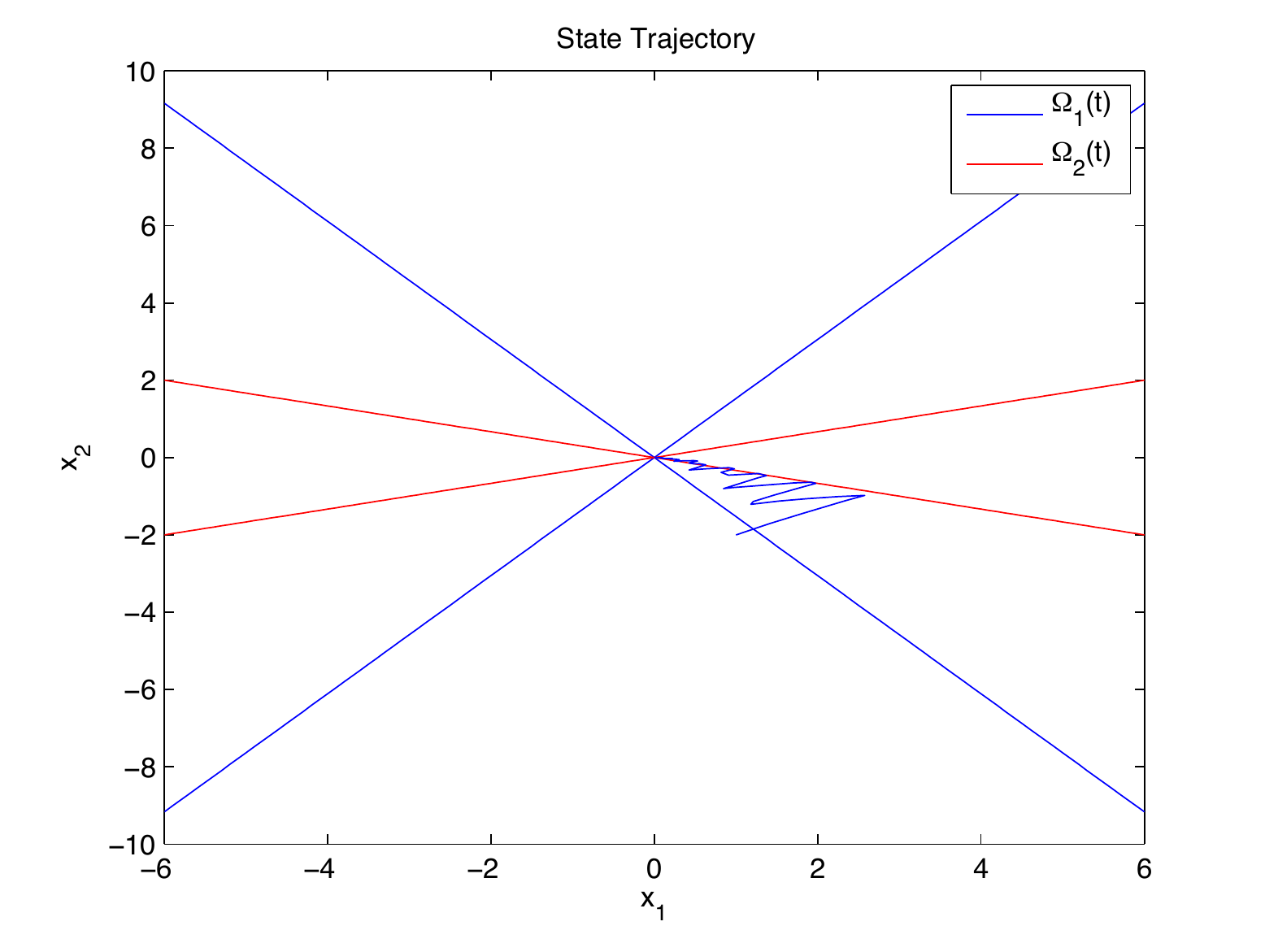}
\end{figure}
Figure \ref{FigTwo} is the state trajectory with initial value of $x_0=[2$ $1 ]^T$. Figure \ref{FigThree} shows the phase portrait. The portrait shows the trajectories moving to the origin in a region where both subsystems decrease the Lyapunov function.

Figures \ref{FigFour} and \ref{FigFive} show the state evolution and phase portraints, respectively, for $B=[1\quad 1]'$. The initial condition is chosen as $x_0=\begin{bmatrix}
2 & -1\end{bmatrix}^T$. Once again, stability is observed via the state-dependent switching strategy.
\begin{figure}[h]
\caption{State Evolution}
\label{FigFour}
\includegraphics[scale=0.6]{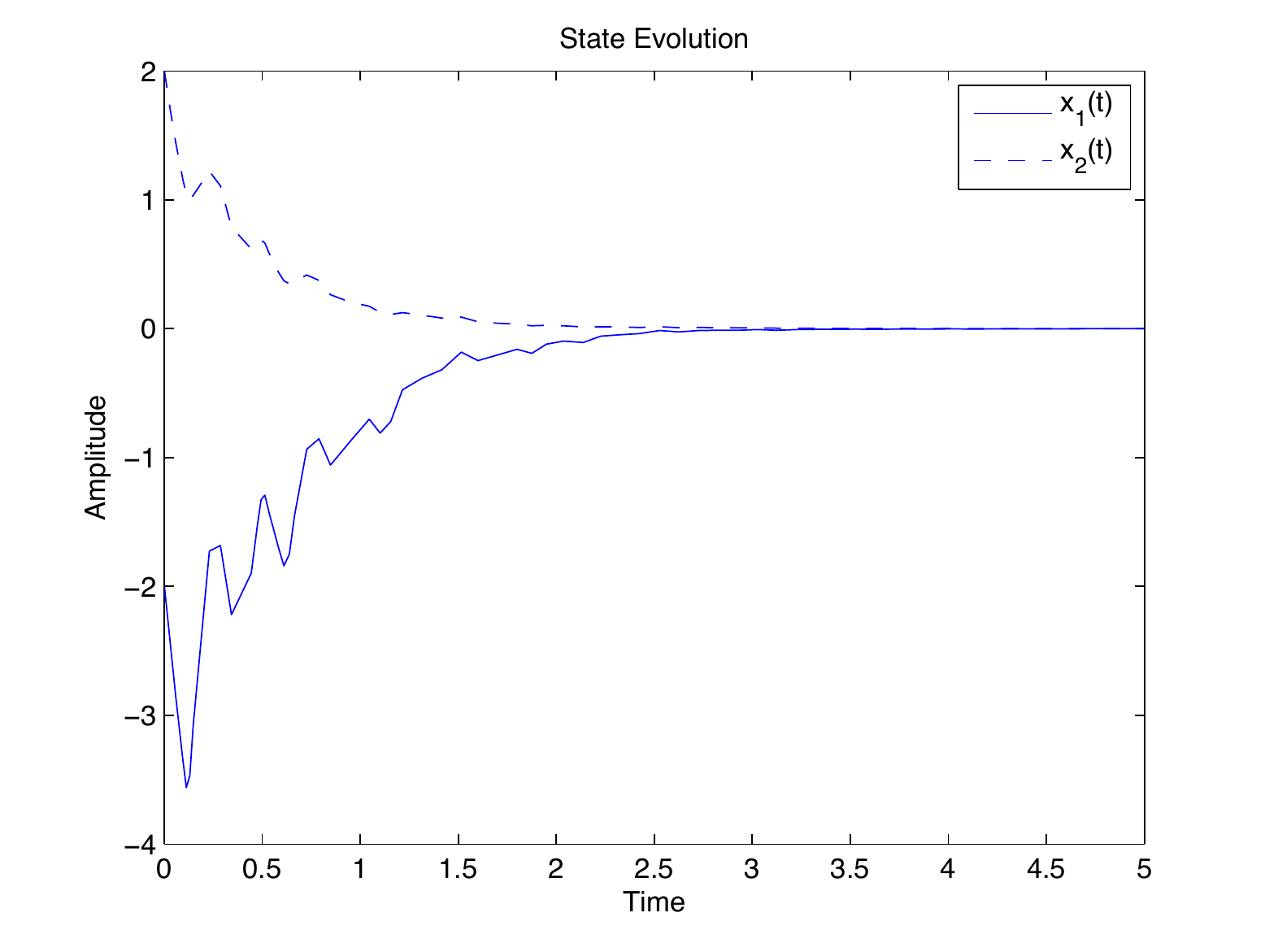}
\end{figure}
\begin{figure}[h]
\caption{Phase Portrait}
\label{FigFive}
\includegraphics[scale=0.6]{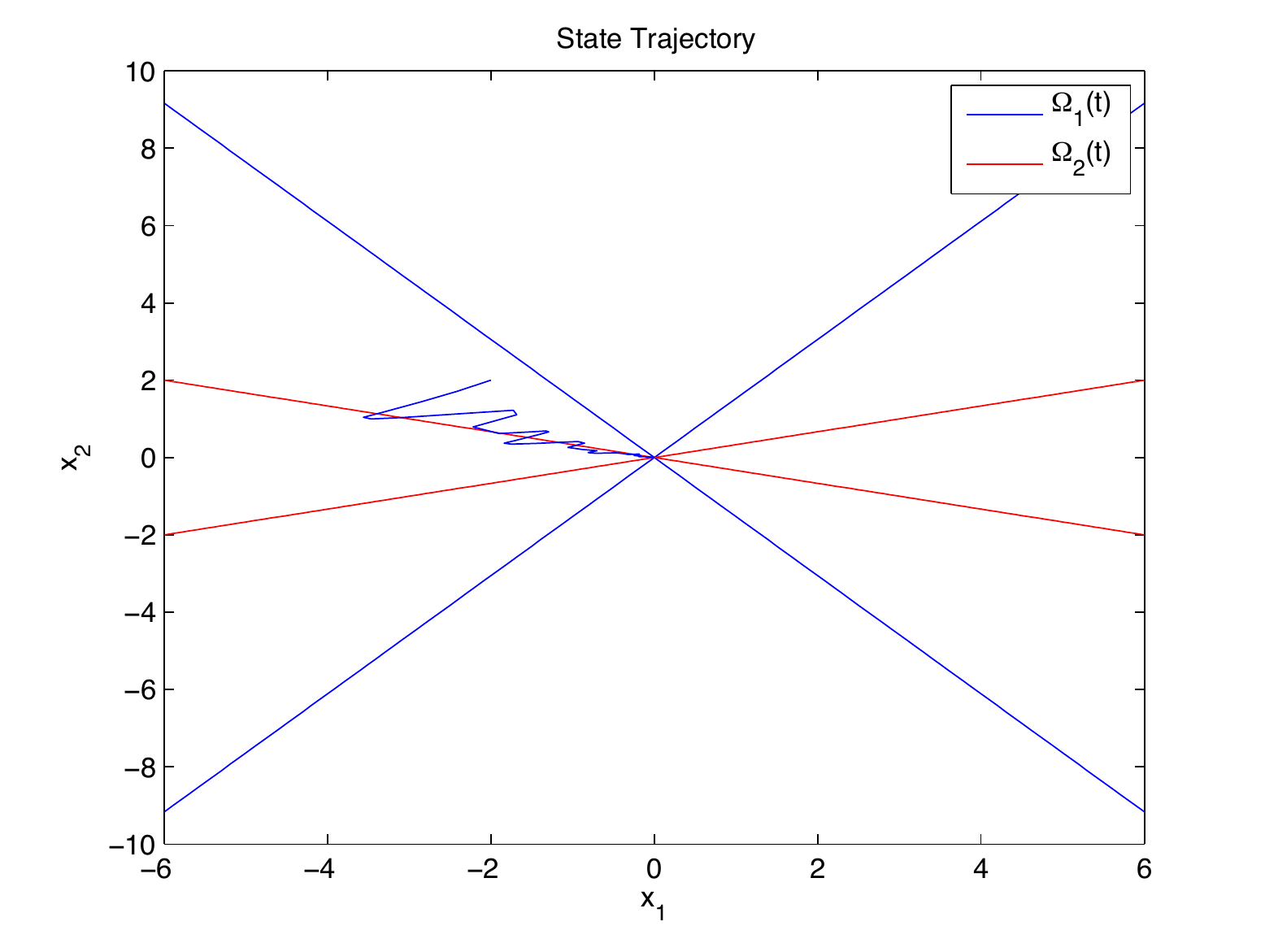}
\end{figure}
\\
\\

Next, we consider the following coupled linear system:
\begin{eqnarray}
A=
\begin{bmatrix}
1 & 1\\
0 & 1\\
\end{bmatrix}
\quad
B=
\begin{bmatrix}
0 \\
1 \\
\end{bmatrix}\nonumber
\end{eqnarray}
The system has both eigenvalues at $\lambda = 1$. The system is stabilizable (in fact, controllable). A linear controller based on the gain
\[
K=
\begin{bmatrix}
8 & 8
\end{bmatrix}
\]
will make $A-BK$ Hurwitz. Repeating, as before, we conclude that if $\alpha_i = \frac{1}{2}$ $\forall i$, $g_1=g_2=8$, and \(\tilde{K} =
\begin{bmatrix}
1 & 1
\end{bmatrix}
\)
we obtain the following equation: 
\[
\tilde{K}\Sigma=
\begin{bmatrix}
8 & 8\\
\end{bmatrix}
\]
Following the same procedure as before, we obtain the following value for $P$.
\[
P=
\begin{bmatrix}
6.5 & 1.75 \\
1.75 & 0.75 \\
\end{bmatrix}
\]
\begin{eqnarray}
\Omega_{1} :&=& (x:x^{T} (A_{1}^{T} P + PA_{1})x<0)\nonumber\\
 &=& ({x:8x_1x_2+5x_2^2-15x_1^2<0})\nonumber\\
\Omega_{2} :&=& ({x:x^{T} (A_{2}^{T} P + PA_{2})x<0})\nonumber\\
 &=& ({x:13x_1^2-8x_1x_2-7x_2^2<0})\nonumber
\end{eqnarray}
$P$ generates two conic regions which determine the switching strategy between the sparsifiers. Figures \ref{FigA} and \ref{FigB} show the state evolution and phase portraits respectively. The initial values is chosen to be $x_0 = [1$ $-0.5]^T$. 
\begin{figure}[!h]
\caption{State Evolution}
\label{FigA}
\includegraphics[scale=0.6]{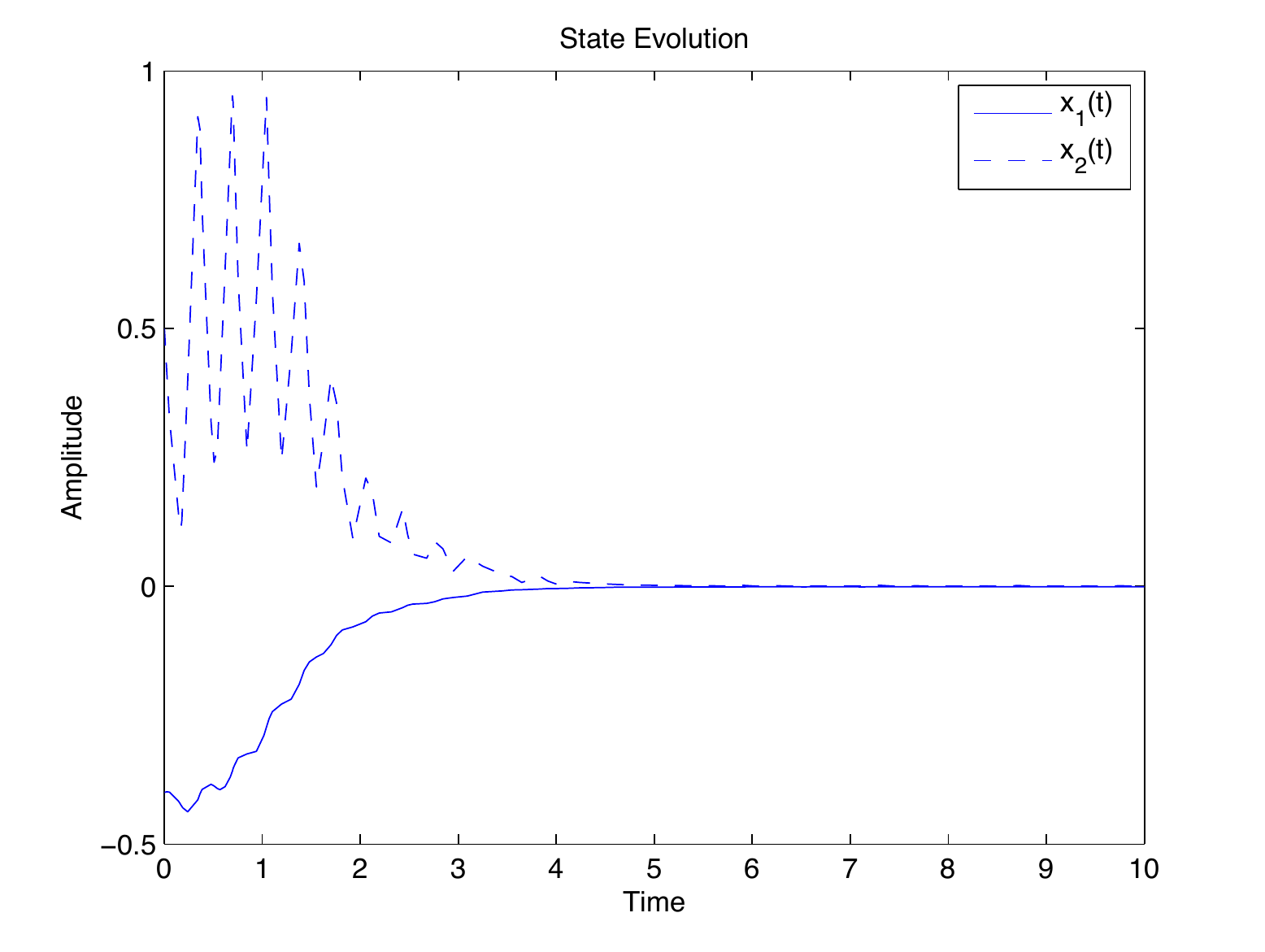}
\end{figure}
\begin{figure}[!h]
\caption{Phase Portrait}
\label{FigB}
\includegraphics[scale=0.6]{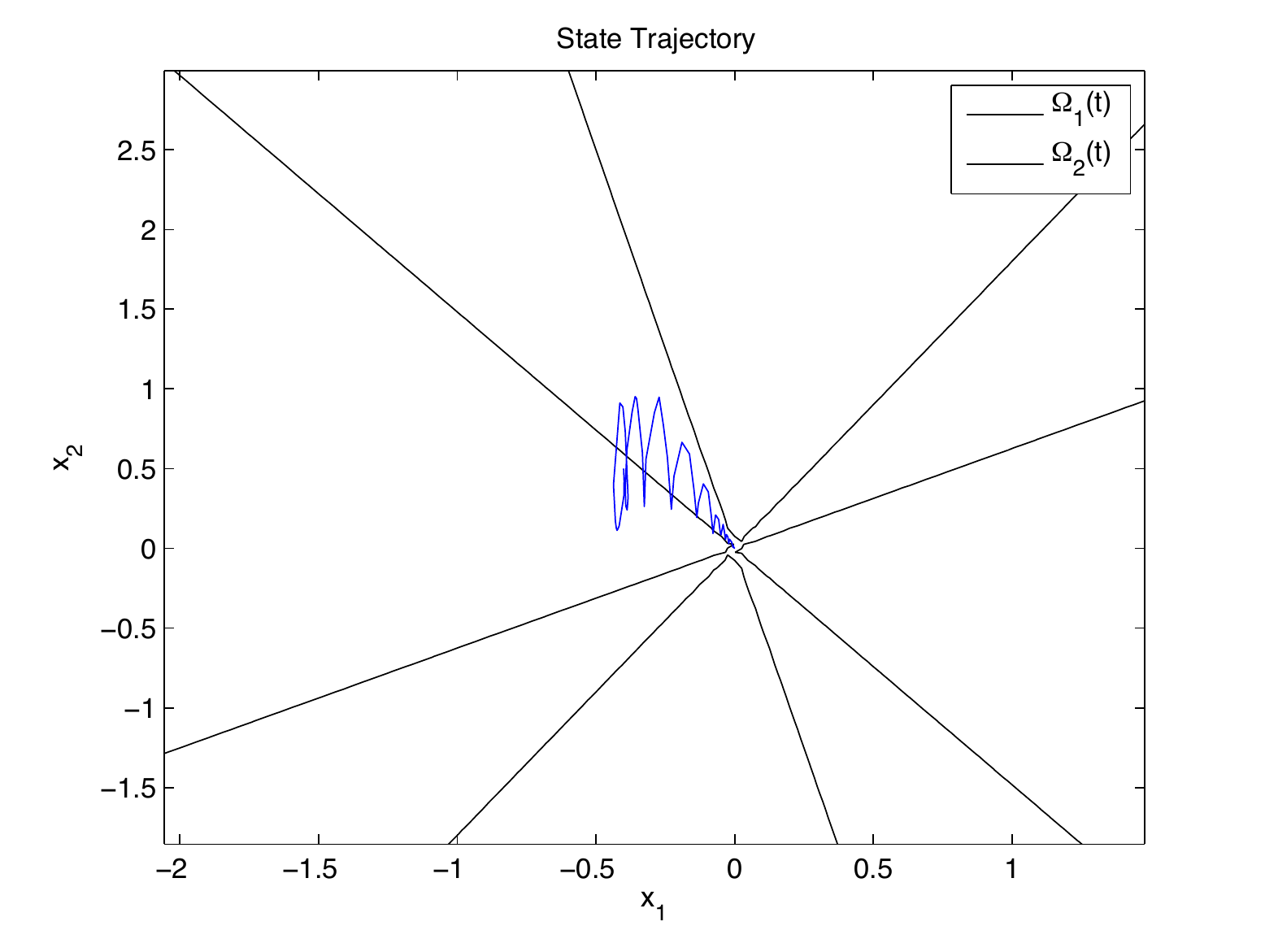}
\end{figure}

\section{Conclusion}

In this paper, we have investigated the effect of switching within a class of sparsifiers, introduced in \cite{bhatta}, on the stability of a linear plant in continuous time settings. Initially, we analyzed the problem of stabilization in low dimensional systems; based on which we addressed the problem of stabilization in systems of arbitrary dimensions. A key contribution in this paper is the construction of a general class of linear non-invertible sparsifiers, and a corresponding switching strategy to stabilize the overall system.

Some of the future directions for research include finding other classes of sparsifiers which might be useful in stabilizing both linear and non-linear dynamical systems. We are currently addressing the problem of limiting the number of switches. We believe that a cost on the number of switches will provide a metric on different classes of sparsifiers, and shed some light on the problem of finding the optimal class.
\bibliographystyle{abbrv}
\bibliography{biblio}
\end{document}